\documentclass[10pt]{amsart}
\usepackage{hyperref}
\usepackage{pb-diagram}
\usepackage{amsmath}
\usepackage{tikz}
\usepackage{paralist}
\usetikzlibrary{matrix,arrows,decorations.pathmorphing}

\numberwithin{equation}{section}
\newtheorem{thm}[equation]{Theorem}
\newtheorem{lem}[equation]{Lemma}
\newtheorem{prp}[equation]{Proposition}

\theoremstyle{definition}
\newtheorem{df}[equation]{Definition}

\theoremstyle{remark}
\newtheorem*{rem}{Remark}

\DeclareMathOperator{\map}{map}

\DeclareMathOperator{\hocolim}{hocolim}

\DeclareMathOperator{\lk}{lk}
\DeclareMathOperator{\Ob}{Ob}

\def\Sp{\mathbf{Top}}

\def\R{\mathbb{R}}
\def\Z{\mathbb{Z}}
\def\N{\mathbb{N}}
\def\P{\mathsf{P}}
\def\V{\mathsf{V}}

\title{On execution spaces of PV-programs}
\author{Krzysztof Ziemia\'nski}
\thanks{}

\begin{document}

\begin{abstract}
	Semaphores were introduced by Dijkstra \cite{D} as a tool for modeling concurrency in computer programs. In this paper we provide a formal definition of PV-programs, i.e.\ programs using semaphores, their state spaces and execution spaces. The main goal of this paper is to prove that every finite homotopy type may appear as a connected component of  the execution space of a PV-program.
\end{abstract}

\maketitle

\section{Introduction}
In 1968 Dijkstra \cite{D} introduced semaphores --- a tool which can be used to synchronize processes in concurrent and distributed systems. For a concurrent program using semaphores (referred further as \emph{a PV-program}) one can assign a space of states in which this program can be during its execution. Such state space carries a structure of d-space \cite{Gr} which determines how its states can change in time. The space of directed paths between the point representing the start of the execution to the final point represents possible executions of the program and will be called the execution space. State spaces of PV-programs are simple examples of Higher Dimensional Automata introduced by Pratt \cite{P}.

The problem of describing the homotopy type of the execution space of a given PV-program was studied intensively in recent years, and many constructions which allow explicit calculations has been presented, see for example \cite{B}, \cite{R1}, \cite{R2}, \cite{R3}, \cite{R4}. It seemed that only particular class of homotopy types can be obtained as execution spaces of PV-programs, for example no examples of execution spaces of PV-programs having torsion in homology was known. In this paper we prove that this conjecture is false; in fact, any finite simplicial complex can be realized as the connected component of the execution space of a PV-program.
 
\subsection*{Organization of the paper}
In Section 2 we provide a strict definition of PV-program, its state space and the execution space. In Section 3 we introduce a notion of execution equivalence --- a relation between PV-programs which preserves the homotopy type of execution spaces. Then, in Section 4,  we define Euclidean complexes and discuss their relationship with state spaces of PV-programs; as a main result we show that every complement-bounded Euclidean complex is a state space of a PV-program. In Section 5 we construct, for any finite simplicial complex $K$, a Euclidean complex having directed path space homotopy equivalent to $|K|$. The  tool we use is the presentation of a directed path space of a Euclidean as a homotopy colimit of smaller spaces presented in \cite{RZ}. In Section 6 we provide an explicit construction of a PV-program having $|K|$ as a connected component of the execution space. Finally, in Section 7 we introduce $PV(n)$-spaces --- topological spaces which can be realized as execution spaces of PV-programs using resources of limited capacity and formulate some open questions.

\subsection*{Notation}
By $\N$ we denote the semiring of non-negative integers. Points of $\R^n$ will be denoted by bold letters, while its coordinates by regular ones with suitable indices; for example $\mathbf{a}=(a_1,\dots,a_n)$. Furthermore, we will write $\mathbf{0}$ for $(0,\dots,0)$ (similarly $\mathbf{1}$, $\mathbf{2}$, \dots). Two kinds of comparators between points of $\R^n$ will be used, namely
\begin{align*}
	\mathbf{a}\leq\mathbf{b}&\Leftrightarrow \forall_{i=1}^n\; a_i\leq b_i\\
	\mathbf{a}<\mathbf{b}&\Leftrightarrow \forall_{i=1}^n\; a_i< b_i.
\end{align*}
We will write $|\mathbf{x}|=\sum_{i=1}^n x_i$ for $\mathbf{x}\in\R^n$. Similarly to one-dimensional case denote $[\mathbf{a},\mathbf{b}]:=\{\mathbf{t}:\; \mathbf{a}\leq\mathbf{t}\leq \mathbf{b}\}$, $(\mathbf{a},\mathbf{b})=\{\mathbf{x}\in\R^n:\; \mathbf{a}<\mathbf{x}<\mathbf{b}\}$ for $\mathbf{a},\mathbf{b}\in\R^n$. Finally, we denote $\lfloor \mathbf{x} \rfloor:=(\lfloor x_1\rfloor, \dots, \lfloor x_n\rfloor)$,  $\lceil \mathbf{x} \rceil:=(\lceil x_1\rceil, \dots, \lceil x_n \rceil)$. If $X$ is a set, then $\chi_A:X\to \{0,1\}$ is a characteristic function of a subset $A\subseteq X$ .

\subsection*{d-spaces}
\emph{A d-space} \cite{Gr} is a pair $(X,\vec{P}(X))$, where $X$ is a topological space, and $\vec{P}(X)\subseteq P(X):=\map([0,1],X)$ is a family of paths on $X$ (called \emph{d-paths}) such that
\begin{itemize}
	\item{all constant paths are d-paths,}
	\item{concatenations of d-paths and d-paths,}
	\item{non-decreasing reparametrizations of d-paths and d-paths.}
\end{itemize}

For $x,y\in X$ denote 
\begin{equation}
	\vec{P}(X)_x^y=\{\alpha\in \vec{P}(X):\; \alpha(0)=x\; \wedge \; \alpha(1)=y\}.
\end{equation}
Important examples of d-spaces are \emph{the directed interval} $\vec{I}=[0,1]$, where $\vec{P}(\vec{I})$ is the family of non-decreasing paths, and \emph{the directed Euclidean space} $\vec{\R}^n$, where
\begin{equation}
	\vec{P}({\vec{\R}^n})=\{\mathbf{f}=(f_1,\dots,f_n)\in P(\R^n):\; \forall_i\; \text{$f_i$ is non-decreasing}\}.
\end{equation}
If $(X,\vec{P}(X))$ is a d-space then every subspace $Y\subseteq X$ carries the restricted structure of a d-space given by $\vec{P}(Y)=\vec{P}(X)\cap P(Y)$. All d-spaces appearing in this paper will be subsets of $\vec{\R}^n$ with the restricted d-structure.

\section{PV-programs}
This Section contains definitions of PV-programs, its state spaces and execution spaces. The original concept of Dijkstra \cite{D} is the following: a PV-program is a family of processes sharing common resources. Every resource has a capacity which determines a number of processes which can simultaneously acquire it. Processes perform simultaneously sequences of operations; every operation is either an acquisition of some resource, or a release of a resource. An acquisition of a resource $r$, denoted by $\P{r}$ causes a process to stop until the resource $r$ becomes available (i.e.\ the number of processes which acquired it becomes less than its capacity); then the process acquires it and continues its execution. A release $\V{r}$ frees a resource $r$ and makes it available for other processes. Processes using semaphores are widely used in practical implementations and its executions were studied from theoretical point of view \cite{FGHR}.

In this paper we consider a slightly general notion of PV-program. Namely, we allow a single operation to acquire and release any number of resources. The execution of such an operation proceeds as follows: first the process releases resources; then it awaits until the resources it needs to acquire became available, and finally acquires them. We will refer to PV-programs using only elementary operations as elementary PV-programs. It turns out (cf.\ \ref{p:Reductions}), \ref{p:ExEq}) that passing from elementary PV-programs to general ones does not produce any new homotopy types of execution spaces.

\begin{df}
	\emph{A resource set} is a finite set $R$ equipped with a capacity function $\mu:R\rightarrow\N\setminus\{0\}$. Elements of a resource set will be called \emph{resources}.
\end{df}
For the remainder of this section let $R$ be a fixed resource set. 

\begin{df}
	\emph{A PV-operation $q$} is a pair of functions $q_{\P},q_{\V}:R\rightarrow \N$. We will also use a notation $q=\V X \P Y$, $X,Y\subseteq R$ whenever
\[
	q_{\V}(r)=\begin{cases}
		1 & \text{for $r\in X$}\\
		0 & \text{for $r\not\in X$}
	\end{cases}
	\qquad
	q_{\P}(r)=\begin{cases}
		1 & \text{for $r\in Y$}\\
		0 & \text{for $r\not\in Y$.}
	\end{cases}
\]	
A PV-operation $\V\emptyset \P\{r\}$ is called \emph{an elementary acquisition} of a resource $r\in R$ and denoted by $\P{r}$. Similarly, $\V\{r\}\P\emptyset$ is called \emph{an elementary release} of $r$ and denoted by $\V{r}$. A PV-operation is elementary if it is either an elementary acquisition or an elementary release. A PV-operation $\emptyset:=\P\emptyset\V\emptyset$ will be called \emph{an empty PV-operation}. 
\end{df}
\begin{df}
	\emph{A PV-process} $Q$ is a sequence of PV-operations $(q^1,\dots,q^l)$. We say that $Q$ is \emph{elementary} if it contains only elementary PV-operations.
\end{df}

\begin{df}
	\emph{A PV-program} $\mathcal{Q}=\{Q_1,\dots,Q_n\}$ is a collection of PV-programs. We say that $\mathcal{Q}$ is elementary if it contains only elementary PV-processes.
\end{df}

For simplicity we will further skip the prefix "PV". In order to define state spaces we need to introduce a notion of progression which intuitively measures how advanced is a process at the moment it performs a given operation.

\begin{df}	
	\emph{A progression} of a process $Q=(q^1,\dots,q^l)$ is a sequence of real numbers $t^1<\dots<t^l$. A progression is \emph{integral} if all numbers $t^i$ are integers.	For a process with progression we will use notation
	\[
		Q=(q^1[t^1], q^2[t^2],\dots, q^l[t^l])=(q^i[t^i])_{i=1}^l.
	\]
	Every process is equipped with \emph{a canonical progression} defined by 
	\[
		Q=(q^1[0], q^2[1],\dots, q^l[l-1])=(q^i[i-1])_{i=1}^l.
	\]	
\end{df}

For the rest of the Section we assume that $Q=(q^i[t^i])_{i=1}^l$ is a process with progression, and $\mathcal{Q}=\{Q_j\}_{j=1}^n$, $Q_j=(q_j^i[t_j^i])_{i=1}^{l_j}$ is a program with progression (i.e.\ all its processes have progressions). Furthermore, we denote $\mathbf{t}^\bot=(t_1^1,\dots,t_n^1)$, $\mathbf{t}^\top=(t_1^{l_1},\dots,t_n^{l_n})$.

\begin{df}
	\emph{The potential function} $a^{Q}_r:\R\rightarrow \Z$ of a process $Q$ for a resource $r$ is defined by
	\[
		a^{Q}_r(t)=\sum_{i: t^i<t} q^i_{\P}(r)-\sum_{i: t^i\leq t} q^i_{\V}(r).
	\]
	\emph{The potential function} $a^{\mathcal{Q}}_r:\R^n\rightarrow \Z$ of a program $\mathcal{Q}$ for a resource $r$ is defined by
	\[
		a^{\mathcal{Q}}_r(\mathbf{t})=a^{\mathcal{Q}}_r(t_1,\dots,t_n):=\sum_{j=1}^n a^{Q_j}_r(t_j).
	\]
\end{df}

The potential function counts how many times a resource $r$ has been acquired by the process (or the processes of the program) when its advancement equals $t$ (resp.\ the advancement of $Q_j$ equals $t_j$). Note that resources are released just before an operation and they are acquired just after an operation. Potential functions are lower semi-continuous and constant on open intervals $(-\infty, t^1)$, $(t^i,t^{i+1})$ for $i=1,\dots,l-1$ and $(t^l,+\infty)$ (in the single process case) or on open hyperrectangles having the form
\[
	\{\mathbf{x}\in\R^n:\; (t_1^{k_1},\dots,t_n^{k_n})<\mathbf{x}<(t_1^{k_1+1},\dots,t_n^{k_n+1})\}
\]
in the case of a program.
	
\begin{df}
	A process $Q$ is \emph{valid} if for every resource $r$
	\begin{itemize}	
	 \item{$\lim_{t\rightarrow+\infty}a^Q_r(t)=0$ (resources are eventually released), and}
	 \item{$a^Q_r(t)\geq 0$ for all $t\in \R$ (resources are acquired before they are released).}
	\end{itemize}
	$Q$ is \emph{elementary valid} if{}f it is elementary, valid, and $a^Q_r(t)\in\{0,1\}$ (which means that no resource is acquired twice at any moment). We say that a program is \emph{valid} (resp.\ \emph{elementary valid}) if all its processes are valid (resp. elementary valid).
\end{df}

Note that if $Q$ is valid, then $a_r^Q(x)=0$ whenever $x\leq t^1$ or $t^l\leq x$.

\begin{df}
	\emph{The state space} of a program $\mathcal{Q}$ is a d-space
	\[
		S(\mathcal{Q}):=\{\mathbf{t}\in \vec{\R}^n:\; \forall_{r\in R}\;\; a^{\mathcal{Q}}_r(\mathbf{t})\leq \mu(r)\}.
	\]
	The d-structure on $S(\mathcal{Q})$ is inherited from $\vec{\R}^n$, i.e.\ d-paths are exactly paths having non-decreasing coordinates.
\end{df}
By lower semi-continuity of potential functions, $S(\mathcal{Q})$ is a closed subset of $\R^n$.

\begin{df}
	\emph{The execution space} of a program $\mathcal{Q}$ with the initial point $\mathbf{a}<\mathbf{t}^\bot$ and the final point $\mathbf{b}>\mathbf{t}^\top$ is the space
	\[
		E(\mathcal{Q},\mathbf{a},\mathbf{b}):=\vec{P}(S(\mathcal{Q}))_{\mathbf{a}}^{\mathbf{b}}.
	\]
\end{df}

\begin{prp}
	The state space and the execution space of a program does not depend, up to homeomorphism, on the choice of progressions and on the choice of an initial and a final point.
\end{prp}
\begin{proof}
	Fix a program with progression $\mathcal{Q}=\{Q_j\}_{j=1}^n$, $Q_j=(q_j^i[t_j^i])_{i=1}^{l_j}$ and let $\bar{\mathcal{Q}}$ denote the same program with another progression, namely $\bar{\mathcal{Q}}=\{\bar{Q}_j\}_{j=1}^n$, $\bar{Q}_j=(q_j^i[\bar{t}_j^i])_{i=1}^{l_j}$. Let $\mathbf{a}<\mathbf{t}^\bot$, $\mathbf{b}>\mathbf{t}^\top$, $\bar{\mathbf{a}}<\bar{\mathbf{t}}^\bot=(\bar{t}_1^1,\dots,\bar{t}_n^1)$ and $\bar{\mathbf{b}}>\bar{\mathbf{t}}^\top=(\bar{t}_1^{l_1},\dots,\bar{t}_n^{l_n})$.
	 For every $j=1,\dots,n$ choose an increasing homeomorphism $\varphi_j:\R\to\R$ such that $\varphi_j(t_j^i)=\bar{t}_j^i$ for $i=1,\dots,l_j$, $\varphi_j(a_j)=\bar{a}_j$ and $\varphi_j(b_j)=\bar{b}_j$.	 
	  Clearly $a_r^{Q_j}(t)=a_r^{\bar{Q}_j}(\varphi_j(t))$ for every resource $r\in R$ and then $a_r^{\mathcal{Q}}(\mathbf{t})=a_r^{\bar{\mathcal{Q}}}(\Phi(\mathbf{t}))$ for $\mathbf{t}\in\R^n$, where $\Phi=\prod_j \varphi_j:\R^n\to\R^n$. As a consequence, $\Phi|_{S(\mathcal{Q})}:S(\mathcal{Q})\to S(\bar{\mathcal{Q}})$ is a d-homeomorphism and then it induces a homeomorphism 
	 \[
	  	E(\mathcal{Q},\mathbf{a},\mathbf{b})=\vec{P}(S(\mathcal{Q}))_{\mathbf{a}}^{\mathbf{b}}\to\vec{P}(S(\bar{\mathcal{Q}}))_{\bar{\mathbf{a}}}^{\bar{\mathbf{b}}}= E(\bar{\mathcal{Q}},\bar{\mathbf{a}},\bar{\mathbf{b}}).\qedhere
	  \]
\end{proof}

We will further omit initial and final points and write $E(\mathcal{Q})$ instead of $E(\mathcal{Q},\mathbf{a},\mathbf{b})$.

\begin{prp}\label{p:ExVal}
	If $\mathcal{Q}$ is valid, then $E(\mathcal{Q})$ is homotopy equivalent to $\vec{P}(S(\mathcal{Q}))_{\mathbf{t}^\bot}^{\mathbf{t}^\top}$.
\end{prp}
\begin{proof}
	Fix $\mathbf{a}<\mathbf{t^\bot}$ and $\mathbf{b}>\mathbf{t^\top}$. Define a map $\mathbf{p}=\prod p_i:\R^n\to \R^n$, where
	\[
		p_j(x)=\begin{cases}
			t^1_j & \text{for $x\leq t^1_j$}\\
			x & \text{for $t^1_j\leq x\leq t_j^{l_j}$}\\
			t_j^{l_j} & \text{for $t_j^{l_j}\leq x$.}
		\end{cases}
	\]
	Note that $a_r^{Q_j}(p_i(x))=a_r^{Q_j}(x)$ for $r\in R$, $x\in\R$, $j=1,\dots,n$. Then $a_r^{\mathcal{Q}}(\mathbf{p}(\mathbf{x}))=a_r^{\mathcal{Q}}(\mathbf{x})$ and therefore $\mathbf{p}(\mathbf{x})\in S(\mathcal{Q})$ if and only if $\mathbf{x}\in S(\mathcal{Q})$. Furthermore, if $\mathbf{x}\in S(\mathcal{Q})$, then also $(1-t)\mathbf{x}+t\mathbf{p}(\mathbf{x})\in S(\mathcal{Q})$ for $t\in[0,1]$, and $\mathbf{p}$ maps d-paths into d-paths. Thus we can define maps
	\[
		F:\vec{P}(S(\mathcal{Q}))_{\mathbf{a}}^{\mathbf{b}}\ni \alpha \mapsto \mathbf{p}\circ\alpha \in \vec{P}(S(\mathcal{Q}))_{\mathbf{t}^\bot}^{\mathbf{t}^\top}
	\]
	and $G:\vec{P}(S(\mathcal{Q}))_{\mathbf{t}^\bot}^{\mathbf{t}^\top}\to \vec{P}(S(\mathcal{Q}))_{\mathbf{a}}^{\mathbf{b}}$, where
	\[
		G(\alpha)(s)=\begin{cases}
			(1-3s)\mathbf{a}+3s\mathbf{t}^\bot & \text{for $s\in[0,\tfrac{1}{3}]$}\\
			\alpha(3s-1) & \text{for $s\in[\tfrac{1}{3}, \tfrac{2}{3}]$}\\
			(3-3s)\mathbf{t}^\top+(3s-2)\mathbf{b} & \text{for $s\in[\tfrac{2}{3},1]$.}\\
		\end{cases}
	\]
	These maps are homotopy inverses --- a homotopy between $F\circ G$ and the identity on $\vec{P}(S(\mathcal{Q}))_{\mathbf{t}^\bot}^{\mathbf{t}^\top}$ is given by
	\[
		H_t(\alpha)(s)=\begin{cases}
			\mathbf{t}^\bot & \text{for $s\leq \tfrac{t}{3}$}\\
			\alpha((s-\tfrac{t}{3})(1-\tfrac{2t}{3})^{-1}) & \text{for $\tfrac{t}{3}\leq s\leq 1-\tfrac{t}{3}$}\\
			\mathbf{t}^\top & \text{for $s\geq 1-\tfrac{t}{3}$,}
		\end{cases}
	\]
	and a homotopy between $G\circ F$ and the identity on $\vec{P}(S(\mathcal{Q}))_{\mathbf{a}}^{\mathbf{b}}$ is given by $H_t(\alpha(s))=\mathbf{a}$ for$s\leq \tfrac{t}{3}$, $H_t(\alpha(s))=\mathbf{b}$ for $s\geq 1-\tfrac{t}{3}$ and
\[
	H_t(\alpha)(s)=	(1-t)\alpha((s-\tfrac{t}{3})(1-\tfrac{2t}{3})^{-1})+t(\mathbf{p}(\alpha((s-\tfrac{t}{3})(1-\tfrac{2t}{3})^{-1})))
\]
	otherwise.
\end{proof}

\section{Execution equivalence}

In this section we introduce a notion of execution equivalence of PV-programs --- an equivalence relation which preserves their execution spaces (up to homotopy equivalence). As before, $R$ stands for a fixed resource set. Given two operations $q,q'$ we define their sum $q+q'$ by $(q+q')_{\P}:=q_{\P}+q'_{\P}$, $(q+q')_{\V}:=q_{\V}+q'_{\V}$.

\begin{df}\label{d:ElOp}
	Let $\sim$ be the equivalence relation on the set of processes using resource set $R$ generated by
	\begin{align*}
		(q^1,\dots,q^{k-1},\emptyset,q^{k},\dots,q^l)& \sim (q^1,\dots,q^{k-1},q^{k},\dots,q^l) \tag{E}\\
		(q^1,\dots,q^{k-1},\V{r},q^{k},q^{k+1},\dots,q^l)& \sim (q^1,\dots,q^{k-1},\V{r}+q^{k},q^{k+1},\dots,q^l)\tag{V}\\
		(q^1,\dots,q^{k-1},q^{k},\P{r},q^{k+1},\dots,q^l)& \sim (q^1,\dots,q^{k-1},q^{k}+\P{r},q^{k+1},\dots,q^l)\tag{P}.
	\end{align*}
	We say that two processes $Q$, $Q'$ are \emph{execution equivalent} if{}f $Q\sim Q'$. Two programs are execution equivalent if there exists a bijection between their processes which maps every process into an execution equivalent one.
\end{df}

\begin{df}
	A process $Q=(q^1,\dots,q^l)$ is \emph{reduced} if{}f for every $i<l$ there exists a resource $r$ such that  $q^i_{\V}(r)>0$ and for every  $i>1$ there exists a resource $r$ such that  $q^i_{\P}(r)>0$.
\end{df}

\begin{prp}\label{p:Reductions}
	Let $Q$ be a process.
	\begin{itemize}
	\item{There exists an elementary process $Q'$ which is execution equivalent to $Q$.}
	\item{There exists a unique reduced process $\tilde{Q}$ which is execution equivalent to $Q$.}
	\end{itemize}
\end{prp}
\begin{proof}
	Assume that $R=\{r_1,\dots,r_s\}$. By using type (E) equivalences we can remove all empty operations. Next, using type (V) and type (P) equivalences we replace every operation $q$ of a process $Q$ by a sequence 
	\[
		\overbrace{\V{r_1},\dots,\V{r_1}}^{\text{$q_{\V}(r_1)$ times}},\dots,\overbrace{\V{r_s},\dots,\V{r_s}}^{\text{$q_{\V}(r_s)$ times}},
		\overbrace{\P{r_1},\dots,\P{r_1}}^{\text{$q_{\P}(r_1)$ times}},\dots,\overbrace{\P{r_s},\dots,\P{r_s}}^{\text{$q_{\P}(r_s)$ times}}.
	\]
	and obtain an elementary process $Q'$ which is execution equivalent to $Q$. Then $Q'$ has the form
	\[
		(S^{\V}_1,S^{\P}_1,S^{\V}_2,S^{\P}_2,\dots;S^{\V}_n,S^{\P}_n),
	\]
	where $S^{\V}_i$ (resp.\ $S^{\P}_i$) are sequences of elementary releases (resp.\ acquisitions) which are non-empty, possibly except $S^{\V}_1$ and $S^{\P}_n$. Now 
	\[
		\tilde{Q}=(\sum_{q\in(S^{\V}_1,S^{\P}_1)}q, \dots, \sum_{q\in( S^{\V}_n;S^{\P}_n)}q)
	\]
	is a reduced program execution equivalent to $Q$. Its uniqueness is clear.
\end{proof}

The following property is a motivation for introducing execution equivalence:

\begin{prp}\label{p:ExEq}
	If two programs are execution equivalent, then their execution spaces are homotopy equivalent.
\end{prp}
\begin{proof}
	It is sufficient to prove this statement for elementary equivalences listed in \ref{d:ElOp}. It is clear for type (E) operations. For type (V) we need to prove that $E(\mathcal{Q})$ and $E(\bar{\mathcal{Q}})$ are homotopy equivalent for programs
	\[\mathcal{Q}=\{Q_1,\dots,Q_n,Q\},\quad \bar{\mathcal{Q}}=\{Q_1, \dots, Q_n,\bar{Q}\}\] where
	\[
		Q=(q^1;\dots;q^i;q^{i+1}; \dots; q^l)\qquad \bar{Q}=(q^1;\dots;q^{i-1};q^i+q^{i+1};q^{i+2};\dots;q^l)
	\]
	and $q^i=\V{r}$ for some resource $r\in R$.
Choose a progression $t^1<\dots<t^{l}$ of $Q$ and a progression
\[t^1<\dots<t^{i-1}<t^{i+1}<t^{i+2}<\dots<t^l\] 
of $\bar{Q}$. Then for every resource $s\neq r$ we have $a^{\mathcal{Q}}_s=a^{\bar{\mathcal{Q}}}_s$, and
\[
	a_r^{\bar{\mathcal{Q}}}(x_1,\dots,x_n,x_{n+1})=\begin{cases}
		a_r^{\mathcal{Q}}(x_1,\dots,x_n,x_{n+1}) & \text{if $x_{n+1}<t^i$ or $t^{i+1}\leq x_{n+1}$} \\
		a_r^{\mathcal{Q}}(x_1,\dots,x_n,x_{n+1})+1 & \text{if $t^i\leq x_{n+1}<t^{i+1}$.} 
	\end{cases}
\]
As a consequence, the identity map on $\R^{n+1}$ restricts to the inclusion $I:S(\bar{\mathcal{Q}})\subseteq S(\mathcal{Q})$. Now let $\varphi:\R\rightarrow \R$ be a non-decreasing map such that $\varphi(x)=x$ for $x\in \R\setminus (t^{i-1},t^{i+1})$, $\varphi(t^i)=t^{i+1}$ and $\varphi(x)<t^{i+1}$ for $x<t^i$. Let
\[
	\Phi:\R^{n+1}\ni (x_1,\dots,x_n,x_{n+1})\mapsto (x_1,\dots,x_n,\varphi(x_{n+1}))\in \R^{n+1}.
\]
Obviously $a_s^{\bar{\mathcal{Q}}}(\Phi(x_1,\dots,x_{n+1}))=a_s^{\mathcal{Q}}(x_1,\dots,x_{n+1})$ for every resource $s$, hence $\Phi(S(\mathcal{Q}))\subseteq S(\bar{\mathcal{Q}})$. Both compositions $\Phi\circ I:S(\bar{\mathcal{Q}})\rightarrow S(\bar{\mathcal{Q}})$ and $I\circ\Phi:S(\mathcal{Q})\rightarrow S(\mathcal{Q})$ are d-homotopic to the identity maps by convex combinations. Therefore they induce homotopy equivalence between execution spaces $E(\mathcal{Q})$ and $E(\bar{\mathcal{Q}})$. An argument for type (P) operations is similar.
\end{proof}

\begin{rem}
	Execution equivalence does not preserve other properties of PV-program as existence of deadlocks. For example, PV-programs
	\[
		((\P{a},\P{b},\V{b},\V{a}), (\P{a},\P{b},\V{b},\V{a}))
	\]
	and
	\[
		((\P{a},\P{b},\V{b},\V{a}),(P{b},\P{a},\V{a},\V{b}))
	\]	
	are execution equivalent but only in the lower one a deadlock can happen.
\end{rem}

\section{Euclidean complexes}

In this section we discuss a relationship between state spaces of PV-programs and Euclidean complexes  --- certain subsets of directed Euclidean space $\vec{\R}^n$.

\begin{df}
\emph{An elementary cube} in $\R^n$ is a subset having the form $[\mathbf{k},\mathbf{l}]$, where $\mathbf{k},\mathbf{l}\in\Z^n$ and $\mathbf{l}-\mathbf{k}\in\{0,1\}^n$. The dimension of a cube is $|\mathbf{l}-\mathbf{k}|$. 
\emph{A Euclidean complex} is a subset $K\subseteq\R^n$ which is a sum of elementary cubes. 
\end{df}

\begin{rem}
	There is an alternative definition of Euclidean complex. Let $A$ be a semi-cubical set defined by $A_0=\Z$, $A_1=\Z$, $A_n=\emptyset$ for $n>1$, and
	\[
		d_1^0(k)=k,\qquad d_1^1(k)=k+1
	\]
	for $k\in A_1$. The geometric realization of $A$ is a real line $\R$, hence the realization of the product $A^n=A\times\dots\times A$ is $\R^n$. Now $K\subseteq \R^n$ is a Euclidean complex if and only if it is the geometric realization of a semi-cubical subset of $A^n$.
\end{rem}

\begin{lem}\label{l:Crit}
	Let $K\subseteq \R^n$ be a subset $K\subseteq \R^n$. The following conditions are equivalent:
	\begin{enumerate}[(a)]
	\item{$K$ is a Euclidean complex.}
	\item{For every $\mathbf{x}\in K$ holds $[\lfloor \mathbf{x} \rfloor, \lceil \mathbf{x} \rceil]\subseteq K$.}
	\item{For every $\mathbf{x}\in \R^n\setminus K$ holds $(\lceil\mathbf{x}-\mathbf{1} \rceil, \lfloor\mathbf{x}+\mathbf{1}  \rfloor)\subseteq \R^n\setminus K$.}
	 \end{enumerate}
\end{lem}
\begin{proof}
	(a)$\Rightarrow$(b).
	Assume that $K$ is a Euclidean complex and that $\mathbf{x}\in K$. There exist $\mathbf{k},\mathbf{l}\in\Z^n$ such that  $\mathbf{x}\in [\mathbf{k},\mathbf{l}]\subseteq K$. Since $\mathbf{k}$ and $\mathbf{l}$ are integral, $\mathbf{k}\leq \lfloor\mathbf{x}\rfloor$ and $\lceil\mathbf{x}\rceil\leq \mathbf{l}$. Thus $[\lfloor \mathbf{x} \rfloor, \lceil \mathbf{x} \rceil]\subseteq K$.\\
	(b)$\Rightarrow$(a). If $K$ satisfies (b), then $K=\bigcup_{\mathbf{x}\in K}[\lfloor \mathbf{x} \rfloor, \lceil \mathbf{x} \rceil]$ is a presentation as a sum of elementary cubes.\\
	(b)$\Rightarrow$(c). Assume that $\mathbf{x}\not\in K$ and that $K\ni \mathbf{y}\in (\lceil\mathbf{x}-\mathbf{1}\rceil,\lfloor\mathbf{x}+\mathbf{1}\rfloor)$. Thus $\mathbf{y}<\lfloor\mathbf{x}+\mathbf{1}\rfloor$ and then $\lfloor\mathbf{y}\rfloor \leq \lfloor\mathbf{x}+\mathbf{1}\rfloor - \mathbf{1} \leq \mathbf{x}$. Similarly we show that $\mathbf{x}\leq \lceil\mathbf{y}\rceil$. Finally, $\mathbf{x}\in[\lfloor\mathbf{y}\rfloor,\lceil\mathbf{y}\rceil]\subseteq K$ which contradicts the assumption.\\
	(c)$\Rightarrow$(b). An argument is similar to the previous one.
\end{proof}

\begin{prp}
	Assume that $\mathcal{Q}=\{Q_j\}_{j=1}^n$ is a program with an integral progression. Then its state space $S(\mathcal{Q})\subseteq \R^n$ is a Euclidean complex.
\end{prp}
\begin{proof}
	For $r\in R$ and $j=1,\dots,n$ the potential function $a^{Q^j}_r$ is lower semi-continuous and constant on open intervals $(k,k+1)$, $k\in\Z$. As a consequence, for any $x\in\R$ and $y\in[\lfloor x\rfloor, \lceil x \rceil ]$ holds $a^{Q_j}_r(x)\geq a^{Q^j}_r(y)$. Then for every $\mathbf{x},\mathbf{y}\in\R^n$ such that $\mathbf{y}\in [\lfloor{\mathbf{x}}\rfloor, \lceil \mathbf{x} \rceil]$ we have $a^{\mathcal{Q}}_r(\mathbf{x})\geq a^{\mathcal{Q}}_r(\mathbf{y})$ and hence $\mathbf{x}\in S(\mathcal{Q})$ implies that $[\lfloor \mathbf{x} \rfloor, \lceil \mathbf{x} \rceil]\subseteq S(\mathcal{Q})$. Then by \ref{l:Crit} $S(\mathcal{Q})$ is a Euclidean complex.
\end{proof}

\begin{prp}\label{p:HolePrp}
	Let $K\subseteq \R^n$ be a Euclidean complex. Assume that the complement $\R^n\setminus K$ is bounded. Then there exist
\begin{itemize}
	\item{a resource set $R$ with all resources having capacity $n-1$,}
	\item{an elementary valid PV-program $\mathcal{Q}=\{Q_1,\dots,Q_n\}$ using $R$,}
	\item{an integral progression of $\mathcal{Q}$,}
\end{itemize}
such that $S(\mathcal{Q})=K$.
\end{prp}

The proof uses the following
\begin{lem}\label{l:Cube}
	Let $K\subseteq \R^n$ be a Euclidean complex having bounded complement. Then there exists a finite set $R$ and a families $\{\mathbf{k}^r\}_{r\in R}$,  $\{\mathbf{l}^r\}_{r\in R}$, $\mathbf{k}^r,\mathbf{l}^r\in\Z^n$, such that
	\[
		K=\R^n \setminus \bigcup_{r\in R}  (\mathbf{k}^r,\mathbf{l}^r). 
	\]
\end{lem}
\begin{proof}
	For every $\mathbf{x}\in \R^n\setminus K$ denote $\mathbf{k}^{\mathbf{x}}:=\lceil \mathbf{x}-\mathbf{1}\rceil$, $\mathbf{l}^{\mathbf{x}}:=\lfloor \mathbf{x}+\mathbf{1}\rfloor$. By \ref{l:Crit} we have $(\mathbf{k}^{\mathbf{x}},\mathbf{l}^{\mathbf{x}})\cap K=\emptyset$, and by boundedness of $\R^n\setminus K$ there is only finitely many hyperrectangles having the form $(\mathbf{k}^{\mathbf{x}},\mathbf{l}^{\mathbf{x}})$.
\end{proof}

\begin{proof}[Proof of {\ref{p:HolePrp}}]
	Choose a set $R$ and families $\mathbf{k}^r$, $\mathbf{l}^r$ from Lemma \ref{l:Cube} and put $\mu(r)=n-1$ for all $r\in R$. Furthermore, let $\mathbf{a}=\min_{r\in R}\mathbf{k}^r$, $\mathbf{b}=\max_{r\in R}\mathbf{l}^r$. For $j=1,\dots,n$ define a process with progression
	\[
		Q_j=(q_j^{a_j}[a_j]; q_j^{a_j+1}[a_j+1]; \dots; q_j^{b_j}[b_j])
	\]
	by putting
	\[
		(q_j^{i})_{\P}(r)=\begin{cases}
			1 & \text{if $k^r_j=i$}\\
			0 & \text{if $k^r_j\neq i,$}
		\end{cases}
		\qquad
		(q_j^{i})_{\V}(r)=\begin{cases}
			1 & \text{if $l^r_j=i$}\\
			0 & \text{if $l^r_j\neq i.$}
		\end{cases}
	\]
	Let $\mathcal{Q}=\{Q_j\}_{j=1}^n$. Note that $a^{\mathcal{Q}}_r(\mathbf{x})\leq n$ for all $\mathbf{x}\in\R^n$ and $a^{\mathcal{Q}}_r(\mathbf{x})=n$ if and only if $\mathbf{x}\in (\mathbf{k}^r,\mathbf{l}^r)$. Finally,
	\[
		S(\mathcal{Q})=\{\mathbf{x}\in \R^n:\; \forall_{r\in R}\; a^{\mathcal{Q}}_r(\mathbf{x})<n\}=\R^n \setminus \bigcup_{r\in R}  (\mathbf{k}^r,\mathbf{l}^r)=K.
	\]
	Since every resource acquires every resource once, then releases it, the program $\mathcal{Q}$ is valid.
\end{proof}

\section{A euclidean complex having a given path space}

In this Section we construct, for any finite simplicial complex $L$, a Euclidean complex $K_L$ such that:
\begin{itemize}
	\item{the complement of $K_L$ is contained in $[\mathbf{0},\mathbf{2}]$,}
	\item{the path space $\vec{P}(K_L)_{\mathbf{0}}^{\mathbf{2}}$ is homotopy equivalent to the geometric realization of $L$.}
\end{itemize}

The main tool we use is the inductive homotopy colimit formula for the space of directed paths on a Euclidean complex described in \cite{RZ}. Let $\Delta^{n-1}$ denote a full simplicial complex with vertices $\{1,\dots,n\}$. We will identify simplices of $\Delta^{n-1}$ (i.e.\ subsets of $\{1,\dots,n\}$) with elements $\mathbf{j}\in \{0,1\}^n$.

\begin{df}
	Let $K\subseteq \R^n$ be a Euclidean complex and let $\mathbf{k}\in \Z^n\cap K$ be a vertex of $K$. \emph{A past link} of $K$ at $\mathbf{k}$, denoted by $\lk^-_{K}(\mathbf{k})$, is the simplicial subcomplex of $\Delta^{n-1}$ defined by the condition
\[
	\mathbf{j} \in \lk^-_K(\mathbf{k})\Leftrightarrow [\mathbf{k}-\mathbf{j},\mathbf{k}]\subseteq K
\]	
for every $\mathbf{j}\in\{0,1\}^n$. 
\end{df}

Let $\mathcal{J}_K^{\mathbf{k}}$ be the inverse category of simplices of $\lk^-_K(\mathbf{k})$. Namely,
\[
	\Ob(\mathcal{J}_K^{\mathbf{k}})=\{\mathbf{j}\in\{0,1\}^n:\; [\mathbf{k}-\mathbf{j},\mathbf{k}]\subseteq K\},
\]
and for every $\mathbf{j}\geq \mathbf{j}'$ there is a single morphism $\mathbf{j}\rightarrow\mathbf{j}'$; if $\mathbf{j}<\mathbf{j}'$ there are no morphisms from $\mathbf{j}$ to $\mathbf{j}'$.  

\begin{thm}\label{t:Dec}
	Let $K\subseteq\R^n$ be a Euclidean complex and let $\mathbf{k}\in\Z^n$ be its vertex. There exists a functor 
	\[
		F_K^{\mathbf{k}}:\mathcal{J}_K^{\mathbf{k}}\rightarrow \Sp
	\]
	and a compatible family of maps (a cocone) $i_{\mathbf{j}}:F_K^{\mathbf{k}}(\mathbf{j})\rightarrow \vec{P}(K)_{\mathbf{0}}^{\mathbf{k}}$ such that the induced map
	\[
		\hocolim_{\mathbf{j}\in \mathcal{J}_K^{\mathbf{k}}}\; F_K^{\mathbf{k}}(\mathbf{j})\rightarrow \vec{P}(K)_{\mathbf{0}}^{\mathbf{k}}
	\]
	is a homotopy equivalence. Furthermore, for every $\mathbf{j}\in \mathcal{J}_K^{\mathbf{k}}$ the space $F_K^{\mathbf{k}}(\mathbf{j})$ is homotopy equivalent to $\vec{P}(K)_{\mathbf{0}}^{\mathbf{k}-\mathbf{j}}$.
\end{thm}
\begin{proof}
	This is a consequence of \cite[Eq. 2.2]{RZ} and \cite[Prop. 2.3]{RZ}.
\end{proof}

We will need only the special case of this statement, when the homotopy colimit reduces to the nerve of the underlying category.
\begin{prp}\label{p:Dec}
	Let $K\subseteq \R^n$ be a Euclidean complex and let $\mathbf{k}\in\Z^n$. Assume that for every $\mathbf{j}\in \{0,1\}^n$ such that $[\mathbf{k}-\mathbf{j},\mathbf{k}]\subseteq K$ the space $\vec{P}(K)_{\mathbf{0}}^{\mathbf{k}-\mathbf{j}}$ is contractible. Then $\vec{P}(K)_{\mathbf{0}}^{\mathbf{k}}$ is homotopy equivalent to $|\lk^-_K(\mathbf{k})|$.
\end{prp}
\begin{proof}
	By \ref{t:Dec} and the Nerve Lemma \cite[Cor.\ 4G.3]{H} we obtain that $\vec{P}(K)_{\mathbf{0}}^{\mathbf{k}}$ is homotopy equivalent to the nerve of the category $\mathcal{J}_K^{\mathbf{k}}$. By applying the Nerve Lemma again to the cover of $\lk^-_K(\mathbf{k})$ by the stars of simplices we obtain a homotopy equivalence $|\lk^-_K(\mathbf{k})|\simeq |N\mathcal{J}_K^{\mathbf{k}}|$.
\end{proof}

\begin{df}
	\emph{A future cone} of a simplicial complex $M\subseteq \Delta^{n-1}$ with apex $\mathbf{k}\in\Z^n$ is a Euclidean complex
	\[
		C^+(\mathbf{k},M)=\bigcup_{\mathbf{j}\in M}[\mathbf{k},\mathbf{k}+\mathbf{j}].
	\]
	Similarly, \emph{a past cone} is
	\[
		C^-(\mathbf{k},M)=\bigcup_{\mathbf{j}\in M}[\mathbf{k}-\mathbf{j},\mathbf{k}].
	\]	
\end{df}

Fix a simplicial complex $L$. Since $L$ can be embedded into a simplex, we will assume that $L\subseteq \Delta^{n-1}$. Define a Euclidean complex $C_L\subseteq [\mathbf{0},\mathbf{1}]\subseteq\R^n$ by
\begin{equation}
	C_L= C^+(\mathbf{0},\partial\Delta^{n-1})\cup C^-(\mathbf{1},L).
\end{equation}

\begin{prp}\label{p:PCL}
	The space $\vec{P}(C_L)_{\mathbf{0}}^{\mathbf{1}}$ is homotopy equivalent to $|L|$.
\end{prp}
\begin{proof}
	For every $\mathbf{1}>\mathbf{j}\in\{0,1\}^n$ the space $	\vec{P}(C_L)_{\mathbf{0}}^{\mathbf{j}}\cong \vec{P}([\mathbf{0},\mathbf{j}])_{\mathbf{0}}^{\mathbf{j}}$ is contractible. The category $\mathcal{J}_{C_L}^{\mathbf{1}}$ is the inverse category of simplices of $L$. Hence by \ref{p:Dec} we have $\vec{P}(C_L)_{\mathbf{0}}^{\mathbf{1}}\simeq |L|$.
\end{proof}

Let
\begin{equation}\label{e:KL}
	K_L:=C_L\cup [\mathbf{1},\mathbf{2}] \cup (\R^n\setminus (\mathbf{0},\mathbf{2})  )
\end{equation}

\begin{prp}\label{p:PKL}
	$\vec{P}(K_L)_{\mathbf{0}}^{\mathbf{2}}\simeq |L|\sqcup S^{n-2}$.
\end{prp}
\begin{proof}
	Let $V$ denote the hyperplane $\{\mathbf{x}\in\R^n:\; x_1+\dots+x_n=1\}$. The section $V\cap K_L$ consists of two disjoint components: a single point $\mathbf{1}$ and $V\cap (\R^n \setminus (\mathbf{0},\mathbf{2}))$. Every path $\alpha\in \vec{P}(K_L)_{\mathbf{0}}^{\mathbf{2}}$ crosses $V\cap K_L$ at a single point; if this point is $\mathbf{1}$, then $\alpha$ is contained $C_L\cup [\mathbf{1},\mathbf{2}]$. If $\alpha$ crosses $V\cap (\R^n \setminus (\mathbf{0},\mathbf{2}))$, then it is contained in $\partial [\mathbf{0},\mathbf{2}]$. Hence
	 \[
		\vec{P}(C_L\cup [\mathbf{1},\mathbf{2}]\cup \partial[\mathbf{0},\mathbf{2}])_{\mathbf{0}}^{\mathbf{2}}=
		\vec{P}(C_L\cup [\mathbf{1},\mathbf{2}])_{\mathbf{0}}^{\mathbf{2}} \sqcup \vec{P}([\partial[\mathbf{0},\mathbf{2}])_{\mathbf{0}}^{\mathbf{2}}
	 \]
	The space $\vec{P}([\partial[\mathbf{0},\mathbf{2}])_{\mathbf{0}}^{\mathbf{2}}$ is homotopy equivalent to $S^{n-2}$ (by \cite[2.6.1]{RZ}). Furthermore,
	\[
		\vec{P}(C_L\cup [\mathbf{1},\mathbf{2}])_{\mathbf{0}}^{\mathbf{2}}\simeq 
		\vec{P}(C_L)_{\mathbf{0}}^{\mathbf{1}}\times \vec{P}([\mathbf{1},\mathbf{2}])_{\mathbf{1}}^{\mathbf{2}}
		\buildrel{\ref{p:PCL}}\over {\simeq} |L|\times \{*\}\simeq |L|.\qedhere
	\]
\end{proof}

As a consequence, we obtain the following
\begin{thm}\label{t:MainThm}
	Let $X$ be a topological space which is homotopy equivalent to the geometric realization of a finite simplicial complex $L$. Then there exists an elementary valid PV-program $\mathcal{Q}$ such that $E(\mathcal{Q})$ has a component homotopy equivalent to $X$.
\end{thm}
\begin{proof}
	By \ref{p:HolePrp} there exists a valid PV-program $\tilde{\mathcal{Q}}$ such that $S(\tilde{\mathcal{Q}})=K_L$. Let $\mathcal{Q}$ be an elementary PV-program which is execution equivalent to $\tilde{\mathcal{Q}}$. Since $\mathcal{Q}$ is valid we have
	\[
		E(\mathcal{Q})\buildrel{\ref{p:ExEq}}\over\simeq E(\tilde{\mathcal{Q}})\buildrel{\ref{p:ExVal}}\over\simeq\vec{P}(S(\tilde{\mathcal{Q}}))_{\mathbf{0}}^{\mathbf{2}}=\vec{P}(K_L)_{\mathbf{0}}^{\mathbf{2}}
		\buildrel{\ref{p:PKL}}\over\simeq |L|\sqcup S^{n-2}.\qedhere
	\]
\end{proof}

\section{An explicit construction}

In this Section we describe an explicit construction of a PV-program such that its execution space contains a connected component homotopy equivalent to a given space. Let $L$ be a finite simplicial complex with vertices $\{1,\dots,n\}$. Let $\{A_r\}_{r\in R'}$ be a family of subsets of $\{1,\dots,n\}$ such that $S$ is a simplex of $L$ if and only if  $A_r\not\subseteq S$ for all $r\in R'$. For every pair $i\neq j\in\{1,\dots,n\}$ define points $\mathbf{c}^{i,j},\mathbf{d}^{i,j}\in \Z^n$ by
\[
	c^{i,j}_m=\begin{cases}0 & \text{for $i\neq m$} \\ 1 & \text{for $i=m$}\end{cases}
	\;\;\;\;\;\;
	d^{i,j}_m=\begin{cases}1 & \text{for $j= m$} \\ 2 & \text{for $j\neq m$}\end{cases}	
\]
and for every $r\in R'$ define $\mathbf{k}^r\in\Z^n$ by
\[
	k^{r}_m=\begin{cases}1 & \text{for $m\in A_r$} \\ 2 & \text{for $m\not\in A_r.$}\end{cases}	
\]
\begin{prp} $K_L=\R^n\setminus U_L$, where
	\[
		U_L=  \left( \bigcup_{i\neq j} \left(\mathbf{c}^{i,j},\mathbf{d}^{i,j}\right) \cup \bigcup_{r\in R'} \left(\mathbf{0},\mathbf{k}^{r}\right)  \right).
	\]	
\end{prp}
\begin{proof}
	We need to prove that every $\mathbf{x}\in\R^n$ is contained in exactly one of the sets $K_L$, $U_L$. Consider the following cases:
	\begin{itemize}
		\item{$\mathbf{x}\not\in(\mathbf{0},\mathbf{2})$. Then $\mathbf{x}\in \R^n\setminus (\mathbf{0},\mathbf{2})\subseteq K_L$ and $\mathbf{x}\not\in U_L$.}
		\item{$\mathbf{x}\in [\mathbf{1},\mathbf{2})$. Then $\mathbf{x}\in K_L$ and $\mathbf{x}\not\in U_L$, since all the points $\mathbf{d}^{i,j}$ and $\mathbf{k}^r$ have at least one coordinate equal $1$ (the family $\{A_r\}_{r\in R'}$ cannot contain an empty set).}
		\item{$\mathbf{x}\in (\mathbf{0},\mathbf{1}]$. Consider the set 
		\[
			J_{\mathbf{x}}=\{m\in\{1,\dots,n\}:\; x_m<1\}.
		\]
		Obviously $K_L\cap (\mathbf{0},\mathbf{1}]=C^-(\mathbf{1},L)$ and $\left(\mathbf{c}^{i,j},\mathbf{d}^{i,j}\right)\cap (\mathbf{0},\mathbf{1}]=\emptyset$ for all $i,j$. Then
		\[
			\mathbf{x}\in K_L\Leftrightarrow \mathbf{x}\in C^-(\mathbf{1},L)\Leftrightarrow J_{\mathbf{x}}\in L
		\]
		\[	
			\mathbf{x}\in U_L\Leftrightarrow \exists_{r\in R'}\; \mathbf{x}\in (\mathbf{0},\mathbf{k}^r)\Leftrightarrow A_r\subseteq J_{\mathbf{x}}.
		\]
		By assumptions, exactly one of these conditions is satisfied.
		}
		\item{$\mathbf{x}\in (\mathbf{0},\mathbf{2})\setminus \big((\mathbf{0},\mathbf{1}]\cup [\mathbf{1},\mathbf{2})\big)$. Then $\mathbf{x}\not\in K_L$. There exists $i,j\in\{1,\dots,n\}$ such that $x_i>1$ and $x_j<1$ and then $\mathbf{x}\in (\mathbf{c}^{i,j},\mathbf{d}^{i,j})$.\qedhere}
	\end{itemize}
\end{proof}

Define a resource set $R=R'\cup \{g_{i,j}\}_{i\neq j}$, $i,j\in\{1,\dots,n\}$.
Now let $\mathcal{Q}(L)=\{Q(L)_1,\dots,Q(L)_n\}$ be a program defined by
\begin{align*}
	Q(L)_m  = & ( \P(R'\cup \{g_{i,j}:\; i\neq m,j \}),\\
	 &  \V(\{r\in R':\; m\in A_r\}\cup \{g_{m,j}:\; m\neq j\})\P\{g_{i,m}:\; i\neq m\},\\
	 & \V(\{r\in R:\; m\not\in A_r\}\cup \{g_{i,j}:\; j\neq m,i\})).
\end{align*}
Choose a progression which assigns respectively 0,1,2 to operations in every process. It follows from the construction presented in the proof of Proposition \ref{p:HolePrp} that 
\[
	S(\mathcal{Q}(L))=\R^n\setminus \left( \bigcup_{i\neq j} \left(\mathbf{c}^{i,j},\mathbf{d}^{i,j}\right) \cup \bigcup_{r\in R'} \left(\mathbf{0},\mathbf{k}^{r}\right)  \right)=K_L
\]
In particular, if we take $n=6$ and 
\begin{multline*}
	\{A_r\}_{r=1}^{10}=\{
		\{1,2,3\},\{2,3,4\},\{3,4,5\},\{4,5,1\},\{5,1,2\},\\ \{1,3,6\},\{2,4,6\},\{3,5,6\},\{4,1,6\},\{5,2,6\}
	\}
\end{multline*}
we obtain a program with 6 processes and 40 resources having multiplicity 5 whose execution space is homotopy equivalent to the disjoint union of a projective plane $\mathbb{R}P^2$ and a sphere $S^4$.

\section{PV-programs with bounded capacity of resources}

The construction of a PV-program such that its execution space contains a given space $X$ as a connected component requires resources of high capacity --- one less than a number of vertices needed for presenting $X$ as a simplicial complex. A natural question arises: what execution spaces can we obtain when we put a restriction on the capacity of resources? This motivates the following definition:

\begin{df}
	We say that a topological space $X$ is \emph{a $PV(n)$-space}, where $n\in\{1,2,\dots\}\cup\{+\infty\}$ if there exists a valid PV-program $\mathcal{Q}$ with the set of resources $R$ such that:
	\begin{itemize}
	\item{the execution space of $\mathcal{Q}$ contains a connected component homotopy equivalent to $X$,}
	\item{all resources $r\in R$ have capacity at most $n$.}
	\end{itemize} 
\end{df}

The main result of this paper states that geometric realizations of finite simplicial complex are $PV(\infty)$-spaces. On the other hand, Raussen \cite{R4} proved that the  space of directed paths on a hyperrectangle with a finite number of hyperrectangular areas removed has a homotopy type of a finite prod-simplicial complex (which can be obviously triangulated). As a consequence, we obtain

\begin{prp}
	A topological space is a $PV(\infty)$-space if and only if it is homotopy equivalent to a geometric realization of a finite simplicial complex.\qed
\end{prp}

Some facts about $PV(n)$-spaces, for $n<\infty$, are known. For example, the Raussen's model gives also the full description of $PV(1)$-spaces --- by \cite[Prop.\ 5.2]{R4}, if $X$ is a $PV(1)$-space, then it has to be contractible. One can also observe that finite products of $PV(n)$-spaces are again $PV(n)$-spaces. Still, there are many interesting open questions:
\begin{enumerate}
	\item{Does there exists a $PV(n)$-space which is not a $PV(n-1)$-space?}
	\item{Is every $PV(2)$-space aspherical, i.e.\ has trivial higher homotopy groups?}
	\item{Does every $PV(n)$-space can be realized as an execution space of a PV-program using only multiplicity $n$ resources?}
	\item{By \cite[2.6.1]{RZ} the sphere $S^{n-1}$ is a $PV(n)$-space. Is this a $PV(k)$-space for some $k<n$?}
\end{enumerate}


\begin{thebibliography}{FGHR}
\bibitem{B}
	P. Bubenik,
	\emph{Simplicial models for concurrency},
	Electronic Notes in Theoretical Computer Science \textbf{283} (2012), 3-12.
\bibitem{D}
	E. W. Dijkstra,
	\emph{Co-operating sequential processes},
	Programming Languages (F. Genuys, ed.), Academic Press, New York, 1968, 43–110.
\bibitem{FGHR}
	L. Fajstrup, E. Goubault, E. Haucourt, and M. Raussen,
	\emph{Components of the fundamental category},
	Appl. Categ. Structures \textbf{12} (2004), 81–108.
\bibitem{Gr}
	M. Grandis,
	\emph{Directed homotopy theory, I. The fundamental category},
	Cahiers Top. Geom. Diff. Categ \textbf{44} (2003), 281-316.
\bibitem{H}
	A. Hatcher,
	\emph{Algebraic Topology},
	Cambridge University Press (2002).
\bibitem{P}
	V. Pratt,
	\emph{Modelling concurrency with geometry},
	 Proc. of the 18th ACM Symposium on Principles of Programming Languages. (1991), 311–322.
\bibitem{R1}	
	M. Raussen,
	\emph{Trace spaces in a pre-cubical complex},
	Topology Appl. 156 \textbf{9} (2009), 1718-1728.
\bibitem{R2}
	M. Raussen,
	\emph{Simplicial models for trace spaces},
	 Algebr. Geom. Topol. \textbf{10} (2010), 1683-1714.
\bibitem{R3}
	M. Raussen,
	\emph{Simplicial models for trace spaces II},
	Algebr. Geom. Topol. \textbf{12} (2012) 1741-1761.
\bibitem{R4}
	M. Raussen,
	\emph{Execution spaces for simple higher dimensional automata},
	Applicable Algebra in Engineering, Communication and Computing
	\textbf{23} (2012), 59-84.
\bibitem{RZ}
	M. Raussen, K. Ziemia\'nski,
	\emph{Homology of spaces of directed paths on Euclidean cubical complexes},
	J. Homotopy Relat. Struct. \textbf{9} (2014), 67-84. DOI 10.1007/s40062-013-0045-4.
\bibitem{Z1}
	 K. Ziemia\'nski,
	 \emph{A cubical model for path spaces in d-simplicial complexes}
	 Topology Appl. \textbf{159} (2012), 2127-2145.
	 
\end{thebibliography}
\end{document}